\documentclass[12pt,a4paper]{amsart}
\usepackage{amsfonts, amsmath, amssymb, amscd, amsthm, graphicx, xspace,url}
\usepackage{hyperref}

\newtheorem{thm}{Theorem}[section]
\newtheorem{cor}[thm]{Corollary}
\newtheorem{lemma}[thm]{Lemma}

\theoremstyle{definition}
\newtheorem{df}[thm]{Definition}
\theoremstyle{remark}

\newcommand{\Z}{\mathbb{Z}}
\newcommand{\N}{\mathbb{N}}

\newcommand{\bS}{\overline{S}}
\newcommand{\bW}{\overline{W}}
\newcommand{\F}{{F}}
\newcommand{\RC}{{RC}}
\newcommand{\St}{{St}}
\newcommand{\T}{\mathcal{T}}
\newcommand{\cc}{{\rm CC}}
\newcommand{\cch}{${\rm CC}_H$}

\begin{document}
\title{One-relator groups with torsion are conjugacy separable}

\author{Ashot Minasyan}
\address[Ashot Minasyan]{School of Mathematics,
University of Southampton, Highfield, Sout\-hampton, SO17 1BJ, United Kingdom.}
\email{aminasyan@gmail.com}

\author{Pavel Zalesskii}
\address[Pavel Zalesskii]{Departamento de Matem\'atica, Universidade de
Bras\'{i}lia, 70910-900 Bras\'{\i}lia-DF, Brazil.}
\email{pz@mat.unb.br}
\thanks{The first author was partially supported by the EPSRC grant EP/H032428/1.}
\date{\footnotesize\today}



\begin{abstract} We prove that one-relator groups with torsion are hereditarily conjugacy separable. Our argument is based on
a combination of recent results of Dani Wise and the first author. As a corollary we obtain that any quasiconvex subgroup of a one-relator group with torsion is also conjugacy separable.
\end{abstract}
\maketitle

\section{Introduction}
Recall that a group $G$ is said to be \textit{conjugacy separable} if for any two non-conjugate elements $x,y \in G$ there is a homomorphism from $G$ to a finite group $M$ such that
the images of $x$ and $y$ are not conjugate in $M$. Conjugacy separability  can be restated by saying that each conjugacy class $x^G=\{gxg^{-1} \mid g \in G\}$
is closed in the profinite topology on $G$. The group $G$ is said to be \textit{hereditarily conjugacy separable} if every finite index subgroup of $G$ is conjugacy separable.
Conjugacy separability is a natural algebraic analogue of solvability of the conjugacy problem in a group and has a number of applications (see, for example, \cite{M-RAAG}).
Any conjugacy separable group is residually finite, but the converse is false. Generally, it may be quite hard to show that a residually finite group is conjugacy separable.

In the recent breakthrough work \cite{Wise-qc-h} Dani Wise proved that one-relator groups with torsion possess so-called quasiconvex hierarchy, and groups with such hierarchy are
virtually compact special.
The class of \textit{special} (or $A$-\textit{special}, in the terminology of \cite{H-W-1}) cube complexes was originally introduced by Fred\'eric Haglund and Dani Wise in \cite{H-W-1},
as cube complexes in which hyperplanes enjoy certain combinatorial properties. They also showed that a cube complex is special if and only if
it admits a combinatorial local isometry to the Salvetti cube
complex (see \cite{Charney}) of some right angled Artin group.
It follows that  the  fundamental group of every special complex $\mathcal{X}$ embeds into
some right angled Artin group.

A group $G$ is said to be \emph{virtually compact special} if $G$ contains a finite index subgroup $P$ such that $P=\pi_1(\mathcal{X})$ for some compact special cube complex $\mathcal X$.
Thus Wise's result implies that any one-relator group $G$, with torsion, is (virtually) a subgroup of a right angled Artin group. In particular, $G$ is
residually finite, which answers an old question of G. Baumslag.

An important fact, established by Haglund and Wise in \cite{H-W-1}, states
that the fundamental group $P$ of a compact special complex is a virtual retract of some finitely generated right angled Artin group.
From the work of the first author \cite{M-RAAG} it follows that $P$ is hereditarily conjugacy separable.
This shows that any one-relator group with torsion possesses a hereditarily conjugacy separable subgroup of finite index.
Unfortunately, in general conjugacy separability is not stable under passing to finite index overgroups (see \cite{Gor}).
The aim of this note is to prove the following:

\begin{thm}\label{thm:main} If $G$ is a one-relator group with torsion then $G$ is hereditarily conjugacy separable.
\end{thm}

This theorem answers positively  Question 8.69  in Kourovka Notebook \cite{K},  posed by  C.Y.~Tang. This question was also raised in \cite{T} in 1982; its special cases have
been considered in \cite{T} and \cite{A-T}.

As a consequence of Theorem \ref{thm:main} we also derive
\begin{cor}\label{cor:qc-cs} If $G$ is a one-relator group with torsion then every quasiconvex subgroup of $G$ is conjugacy separable.
\end{cor}

Our proof of Theorem \ref{thm:main} uses the above mentioned results of Wise, Haglund-Wise and the first author, and employs the
quasiconvex hierarchy for one-relator groups with torsion, that was investigated by Wise in \cite{Wise-qc-h}.

\section{Background on one-relator groups with torsion}\label{sec:backgr}
Let
\begin{equation}\label{eq:G}
G=\langle S \,\|\, W^n \rangle
\end{equation}
be a \emph{one-relator group with torsion}, where $S$ is a finite alphabet,
$n \ge 2$ and  $W$ is a cyclically reduced word, which is not a proper power in the free group $\F(S)$.

Newman's spelling theorem \cite[Thm. 3]{Newman} (see also \cite[IV.5.5]{L-S}) implies that every freely reduced word over $S^{\pm 1}$, representing the identity element of $G$, contains a subword of $W^n$
of length strictly greater than $(n-1)/n$ times the length of $W^n$. Since $(n-1)/n \ge 1/2$ it follows that the presentation \eqref{eq:G} satisfies Dehn's algorithm (\cite[IV.4]{L-S}); in particular $G$ has a
linear Dehn function, and hence it is word hyperbolic. For the background on hyperbolic groups and quasiconvex subgroups the reader is referred to \cite{Mih}.

Another important fact, proved by Newman in \cite[Thm. 2]{Newman} (see also \cite[p. 956]{K-S}), states that centralizers of non-trivial elements in one-relator groups with torsion are cyclic.

Many results about one-relator groups are proved using induction on some complexity depending on the word $W$. In this paper we will use the \emph{repetition complexity} $\RC(W)$ of $W$
employed by Wise in \cite{Wise-qc-h}. This is defined as the difference between the length  of $W$, and the number of distinct letters from $S$ that occur in $W$.
For example, if $S=\{a,b,c\}$ then $\RC(ab^2a^{-1}c^{-3})=7-3=4$.

Start with a one relator-group $G$ given by presentation \eqref{eq:G}.
Recall that a \emph{Magnus subgroup} $M$ of $G$ is a subgroup generated by a subset $U\subset S$ such that $U$ omits at least one generator appearing in $W$.
By the famous Magnus's Freiheitssatz, $M$ is free and $U$ is its free generating set.

Observe that if $\RC(W)=0$ then every letter appears in $W$ exactly once. In this case, using Tietze transformations,  it is easy to see that $G$ is isomorphic to the free product of a free group of rank $|S|-1$
with the cyclic group of order $n$.

Assume, now, that $\RC(W)>0$. Then, following  \cite[18.2]{Wise-qc-h}, one can let $H=G* \langle t \rangle$, and represent $H$ as an HNN-extension of another one-relator group
$K=\langle \bS \,\|\, \bW^n\rangle$, where $|\bS|<\infty$, $\bW$ is some cyclically reduced word in the free group $\F(\bS)$, and
the associated subgroups are Magnus subgroups $M_1,M_2$ of $K$. In other words,
there are subsets $U_1,U_2 \subset \bS$, each of which omits some letter of $\bW$, and a bijection $\alpha: U_1 \to U_2$ such that $M_i=\langle U_i \rangle$, $i=1,2$, and
$H$ has the presentation
\begin{equation}\label{eq:H}
H= \langle  \bS,t \,\|\, \bW^n, tut^{-1}=\alpha(u) \mbox{ for all }u \in U_1\rangle.
\end{equation}

Moreover, in  \cite[18.3]{Wise-qc-h} Wise shows that one can do this in such a way that $\RC(\bW)<\RC(W)$.

\begin{lemma}\label{lem:hcs_sbgp} The group $H$ defined above contains a finite index normal subgroup $L \lhd H$ such that $L$ is hereditarily conjugacy separable.
\end{lemma}

\begin{proof} In \cite[Ch.  18]{Wise-qc-h} Wise shows that $H$ is virtually compact special. By the work of Haglund and Wise
from \cite[Ch. 6]{H-W-1}, $H$ contains a finite index subgroup $L$ such that $L$ is a virtual retract of some finitely generated right angled Artin group $A$.
Now, a result of the first author \cite[Cor. 2.1]{M-RAAG} implies that $L$ is hereditarily conjugacy separable.
\end{proof}

The next statement follows from a combination of results of Wise \cite{Wise-qc-h} and Haglund-Wise \cite{H-W-1}:

\begin{lemma} \label{lem:fi-sep} Let $P$ be a finite index subgroup of $K$ or $M_1$, or $M_2$. Then $P$ is closed in the profinite topology of $H$.
\end{lemma}

\begin{proof} The group $H$ is hyperbolic as a free product of two hyperbolic groups, and
by \cite[Lemma 18.8]{Wise-qc-h} $K$, $M_1$ and $M_2$ are all quasiconvex subgroups of $H$. Since a finite index subgroup of a quasiconvex subgroup is itself quasiconvex, it follows that
$P$ is quasiconvex in $H$.

As we already mentioned above, \cite[Cor. 18.3]{Wise-qc-h} states that $H$ is virtually compact special. Now we can use \cite[Thm. 7.3, Lemma 7.5]{H-W-1}, which imply that
any quasiconvex subgroup of $H$ is separable in $H$. Thus the lemma is proved.
\end{proof}

\section{Some auxiliary facts}
First let us specify some notation. If $A$ is a group and $C,D \subseteq A$, then $C^D$ will denote the subset defined by
$C^D=\{dcd^{-1}\mid c \in C,d \in D\}$. If $x \in A$ and $E\leqslant A$ then $C_E(x)=\{g \in E \mid gx=xg\}$ will denote the centralizer of $x$ in $E$.

Recall that a subset $C$ of a group $A$ is said to be \emph{separable} if $C$ is closed in the profinite topology of $A$. This is equivalent to the following property:
for every $y \in A \setminus C$  there exist a finite group $Q$ and an epimorphism $\psi:H \to Q$ such that $\psi(y)\notin \psi(C)$ in $Q$.

The following notion is helpful for proving hereditary conjugacy separability of groups. It is similar to \cite[Def. 3.1]{M-RAAG}.

\begin{df}\label{df:CCH} Let $H$ be a group and $x \in H$. We will say that the element $x$
satisfies the \textit{Centralizer Condition  in $H$} (briefly,
$\cc_H$), if for every finite index  normal subgroup $P \lhd H$
there is a finite index normal subgroup $N \lhd H$ such that $N
\leqslant P$ and $ C_{H/N} (\psi(x))  \subseteq \psi\left (C_H(x)
P\right)$ in $H/N$, where $\psi: H \to H/N$ is the natural
homomorphism.
\end{df}

The condition \cch{} defined above is actually quite natural from the viewpoint of the profinite completion $\widehat H$ of $H$.
Indeed, in \cite[Prop. 12.1]{M-RAAG} it is shown that if $H$ is residually finite then $x \in H$ has \cch{} if and only if $C_{\widehat{H}}(x)=\overline{C_{H}(x)}$,
where the right-hand side is the closure of $C_H(x)$ in the profinite completion $\widehat H$.

The next two lemmas were proved by the first author in \cite[Lemmas 3.4 and 3.7]{M-RAAG}. The first one shows why the Centralizer Condition is useful, and the
second lemma provides a partial converse to the first one.

\begin{lemma}\label{lem:CCH->sep_c_c_for_sbgps} Suppose that $H$ is a group, $H_1 \leqslant H$ and $x \in H$.
Assume that the element $x$ satisfies {\cch} and the conjugacy class $x^H$ is separable in $H$.
If the double coset $C_H(x)H_1$ is separable in $H$, then the $H_1$-conjugacy class $x^{H_1}$ is also
separable in $H$.
\end{lemma}

\begin{lemma}\label{lem:for_a_given_P_sep_cc->CCH} Let $H$ be a group.
Suppose that $x \in H$, $P\lhd H$ and $|H:P|<\infty$. If the
subset $x^{P}$ is separable in $H$, then there is a finite index normal subgroup $N \lhd H$
such that $N \leqslant P$ and $ C_{H/N} (\psi(x))  \subseteq \psi\left (C_H(x) P\right)$
in $H/N$ (where $\psi:H \to H/N$ denotes the natural homomorphism).
\end{lemma}

The proof of Theorem \ref{thm:main} will also use the following two auxiliary statements.

\begin{lemma} \label{lem:conj_in_HNN} Let $A$ be a group  and let $C_1,C_2 \leqslant A$ be isomorphic subgroups with a fixed isomorphism $\varphi:C_1 \to C_2$.
Let $B=\langle A,t \,\|\, tgt^{-1}=\varphi(g) \mbox{ for all }g \in C_1\rangle$ be the corresponding HNN-extension of $A$. Suppose that $x,y \in A$ are elements such that
$y \notin x^A$ and $x \notin C_i^A$ for $i=1,2$. Then $y \notin x^B$ and $C_B(x)=C_A(x)$ in $B$.
\end{lemma}

\begin{proof} Let $\T$ be the Bass-Serre tree associated to the splitting of $B$ as an HNN-extension of $A$. Then $x$ fixes a particular vertex $v$ of $\T$, where the stabilizer $\St_B(v)$ of $v$ in $B$ is equal to $A$.
The stabilizer of any edge $e$, adjacent to $v$, is
$C_i^a$ for some $i \in \{1,2\}$ and some $a \in A$ (see \cite{Serre}). Therefore, the assumptions imply that $x$ does not fix any edge of $\T$ adjacent to $v$. Since the fixed point set of an isometry of a tree
is connected, it follows that $v$ is the only vertex of $\T$ fixed by $x$.

Arguing by contradiction, suppose that $y \in x^B$, thus there is
$b \in B$ such that $y=bxb^{-1}$ in $B$. Then $b \circ v$ is the
only vertex of $\T$ fixed by $y$. Since $A=\St_B(v)$ and $y \in A$
the latter implies that $b \circ v=v$. Hence $b \in \St_B(v)=A$,
i.e., $y \in x^A$, contradicting  one of the assumptions. Thus $y
\notin x^B$, as claimed.

For the final assertion, suppose that $b \in C_B(x)$, i.e., $x=bxb^{-1}$. The same argument as above shows that $b \in A$, hence $b \in C_A(x)$.
\end{proof}

\begin{lemma}\label{lem:free_sbgp-cd} Let  $A$ be a  group with a free subgroup $F \leqslant A$ and let $g \in A\setminus\{1\}$ be an element of finite order.
Suppose that every finite index subgroup of $F$ is separable in $A$.
Then there exists a finite index normal subgroup $N \lhd A$ such that $\psi(g) \notin \psi(F)^{A/N}$, where $\psi:A \to A/N$ denotes the natural epimorphism.
\end{lemma}

\begin{proof} Since every finite index subgroup of $F$ is separable in $A$ and $F$ is residually finite, the assumptions imply that
$A$ is residually finite and the profinite topology of $A$ induces the full profinite
topology on $F$. Therefore by Lemma 3.2.6 in \cite{RZ}  the
closure $\overline F$, of $F$, in the profinite completion $\widehat A$, of $A$,
is naturally isomorphic to the profinite completion $\widehat F$ of $F$. Then
in the profinite completion $\widehat A$, of $A$, the claim of the lemma reads as
follows: $g$ is not conjugate to $\overline F\cong \widehat F$ in  $\widehat A$.
Indeed, $\overline F^{\widehat A}=\lim\limits_{\displaystyle\longleftarrow}\psi_N(F)^{A/N}$, where $\psi_N:F \to F/N$ denotes the natural epimorphism and
the inverse limit is taken over the directed set of all finite index
normal subgroups $N\triangleleft_f A$. Therefore
$\psi_N(g) \notin \psi_N(F)^{A/N}$ for some $N\triangleleft_f A$ if and only if
$g\notin \overline{F}^{\widehat A}$. But $\overline{F}\cong \widehat F$ is torsion-free by Proposition 22.4.7 in \cite{FJ}, hence the result follows.
\end{proof}

\section{Proofs}
\begin{proof}[Proof of Theorem \ref{thm:main}] Let $G$ be a one-relator group given by the presentation \eqref{eq:G}.
The result will be proved by induction on $\RC(W)$. If $\RC(W)=0$ then $G$ is isomorphic to the free product ${F}_{m}*\Z/n\Z$, where $m=|S|-1$ and $F_m$ is the free group of rank $m$.
Therefore $G$ is virtually free and so it is hereditarily conjugacy separable by Dyer's theorem \cite{Dyer}.

Thus we can further assume that $\RC(W)>0$. Let $H\cong G *\Z$, $K$, $M_1$, $M_2$, $U_1$, $U_2$  and $\alpha:M_1 \to M_2$ be as described in Section \ref{sec:backgr}.
Then $K=\langle \bS \,\|\, \bW^n\rangle$, where $RC(\bW)<\RC(W)$, and so
$K$ is hereditarily conjugacy separable by induction. Since $G$ is a retract of $H$, to prove the theorem it is enough to show that $H$ is hereditarily conjugacy separable (cf. \cite[Lemma 9.5]{M-RAAG}).

Observe that $H$ is itself a one-relator group with torsion. Therefore, by Newman's theorem \cite[Thm. 2]{Newman}, centralizers of non-trivial elements in $H$ are cyclic.
We also recall that, according to Lemma \ref{lem:hcs_sbgp}, $H$ contains a finite index normal subgroup $L$ which is hereditarily conjugacy separable. 

Let $H_1 \leqslant H$ be an arbitrary finite index subgroup and let $x \in H$ be an arbitrary element. We will show that the subset $x^{H_1}$ is separable in $H$
by considering two different cases.

\noindent\emph{Case 1:} $x$ has infinite order in $H$. Since $L$
is hereditarily conjugacy separable, $L_1=H_1\cap L$ is a normal
conjugacy separable subgroup of finite index in $H$. Set
$l=|H:L_1|$. Then $x^l \in L_1\setminus \{1\}$ and $C_H(x^l)$ is
infinite cyclic. It follows that for any $y \in H\setminus
x^{H_1}$, $y^l \notin (x^l)^{H_1}$. Indeed, if $x^l=h y^l h^{-1}$
for some $h \in H_1$, then both $x$ and $hyh^{-1}$ belong to the
infinite cyclic subgroup $C_H(x^l)$. But in the infinite cyclic
group any element can have at most one $l$-th root, thus
$x=hyh^{-1}$, contradicting the assumption that $y \notin
x^{H_1}$.

Since $L_1$ is conjugacy separable, $(x^l)^{L_1}$ is closed in the profinite topology of $L_1$, and since $|H:L_1|<\infty$ this implies that
$(x^l)^{L_1}$ is separable in $H$. Moreover, we can also deduce that the subset $(x^l)^{H_1}$ is separable in $H$, because it equals to
a finite union of conjugates of $(x^l)^{L_1}$, as $L_1$ has finite index in $H_1$. Since $y^l \notin (x^l)^{H_1}$, there are a finite group $Q$ and an epimorphism $\psi:H \to Q$
such that $\psi(y^l) \notin \psi\left((x^l)^{H_1}\right)=\left(\psi(x)^l\right)^{\psi(H_1)}$. Therefore $\psi(y) \notin \psi(x^{H_1})$ in $Q$, as required. Thus $x^{H_1}$ is separable in $H$.

\noindent\emph{Case 2:} $x$ has finite order in $H$. Note that we can assume that $x \neq 1$ in $H$ because otherwise $x^{H_1}=\{1\}$ is separable in $H$ as $H$ is residually finite
(by Wise's work \cite{Wise-qc-h} $H$ possesses a finite index subgroup that embeds into a right angled Artin group, and right angled Artin groups are
well-known to be residually finite). Now we are going to verify that all the assumptions of Lemma~\ref{lem:CCH->sep_c_c_for_sbgps} are satisfied.

\textbf{Claim I:} the conjugacy class $x^H$ is separable in $H$.

By the torsion theorem for HNN-extensions (\cite[IV.2.4]{L-S}), $x \in K^H$. Thus, without loss of generality, we can assume that $x \in K$.

Consider any element $y \in H\setminus x^H$. If $y$ has infinite order then, since $H$ is residually finite, there is a finite group $Q$ and an epimorphism $\psi: H \to Q$, such that the order of $\psi(y)$ in $Q$ is greater than the order of $x$ in $H$
(and, hence, of $\psi(x)$ in $Q$). It follows that $\psi(x)$ is not conjugate to $\psi(y)$ in $Q$.

Thus we can further suppose that $y$ also has finite order in $H$; as before this allows us to assume that $y \in K$. Consequently $y \in K \setminus x^K$, and by conjugacy separability of $K$,
we can find a finite index normal subgroup $K_0 \lhd K$ such that the images of $x$ and $y$, under the natural epimorphism $K \to K/K_0$, are not conjugate in $K/K_0$.

According to Lemmas \ref{lem:fi-sep} and \ref{lem:free_sbgp-cd},
$H$ contains finite index normal subgroups $N_1,N_2 \lhd H$ such that the image of $x$ in $H/N_i$ is not conjugate to the image of $M_i$ for $i=1,2$.
By Lemma \ref{lem:fi-sep}, $K_0$ is separable in $H$, hence there exists a finite index normal subgroup $N_0\lhd H$ such that $N_0 \cap K \subseteq K_0$.
Let $N'\lhd H$ and $K_1 \lhd K$ denote the finite index normal subgroups of $H$ and $K$ respectively, defined by $N'=N_0\cap N_1\cap N_2$ and $K_1=K \cap N'$.

Let $\xi:K \to K/K_1$ denote the natural epimorphism. Note that the isomorphism $\alpha:M_1\to M_2$ gives rise to the isomorphism $\bar\alpha:\xi(M_1) \to \xi(M_2)$, defined by
$\bar\alpha(\xi(g))=\xi(\alpha(g))$ for all $g \in M_1$. Indeed, the fact that $\bar\alpha$ is well-defined is essentially due to the construction of
$K_1$ as the intersection of $K$ with a normal subgroup $N'$ of $H$, and so $\xi$ is a
restriction to $K$ of $\tilde{\xi}:H \to H/N'$. Thus for any $g,h \in M_1$ with $\xi(g)=\xi(h)$ we have
$$\bar\alpha(\xi(g))=\tilde{\xi}(\alpha(g))=\tilde{\xi}(t g t^{-1})=\tilde{\xi}(t)\tilde{\xi}(g)\tilde{\xi}(t^{-1})=\tilde{\xi}(tht^{-1})=\tilde{\xi}(\alpha(h))=\bar\alpha(\xi(h)).$$

Let $\bar H$ be the HNN-extension of $K/K_1$ with associated subgroups $\xi(M_1)$ and $\xi(M_2)$,
defined by $$\bar H=\langle K/K_1, \bar t \,\|\, \bar t \xi(u) \bar t^{-1}=\bar\alpha(\xi(u)) \mbox{ for all } u \in U_1\rangle.$$
Note that $\bar H$ is virtually free since $|K/K_1|<\infty$ (see, for example, \cite[II.2.6, Prop. 11]{Serre}).
Clearly $\xi$ extends to a homomorphism $\eta:H \to \bar H$, given by $\eta(t)=\bar t$ and $\eta(g)=\xi(g)$ for all $g \in K$.

Let us show that $\eta(x)=\xi(x)$ is not conjugate to $\eta(y)=\xi(y)$ in $\bar H$.  Indeed, $\xi(y)\notin \xi(x)^{K/K_1}$ because the homomorphism $K\to K/K_0$ factors through $\xi$ by
construction (as $K_1=K \cap N' \subseteq K \cap N_0 \subseteq K_0$) and the images of $x$ and $y$ are not conjugate in $K/K_0$.
On the other hand, since $K_1 \subseteq N_1 \cap N_2$, we have $\xi(x) \notin \xi(M_i)^{K/K_1}$ for $i=1,2$. Therefore, $\xi(y) \notin \xi(x)^{\bar H}$ by Lemma \ref{lem:conj_in_HNN}.

It remains to recall that $\bar H$ is conjugacy separable by Dyer's theorem \cite{Dyer}, and so there exist a finite group $Q$ and a  homomorphism $\zeta:\bar H\to Q$ such that
$\zeta(\eta(y)) \notin \zeta(\eta(x))^Q$ in $Q$. Hence the homomorphism $\psi=\zeta \circ \eta:H \to Q$ distinguishes the conjugacy classes of $x$ and $y$, as required.
Thus we have shown that $x^H$ is separable in $H$.

\textbf{Claim II:} $x$ satisfies the Centralizer Condition {\cch} from Definition \ref{df:CCH}.

This will be proved similarly to Claim I. As above, without loss of generality, we can assume
that $x \in K$. Consider any finite index normal subgroup $P\lhd H$ and let $R=K \cap P$.

Since $K$ is hereditarily conjugacy separable by induction, the finite index subgroup $E=R\langle x \rangle \leqslant K$ is conjugacy separable.
Hence the subset $x^{E}=x^R$ is separable in $E$.
And since $|K:E|<\infty$ we see that $x^R$ is separable in $K$. Therefore we can apply Lemma \ref{lem:for_a_given_P_sep_cc->CCH} to find a finite index normal subgroup $K_0\lhd K$ such
that $K_0\leqslant R$ and the centralizer of the image of $x$  in $K/K_0$ is contained in the image of $C_K(x)R$ in $K/K_0$.

Arguing as in Claim I, we can choose finite index normal subgroups $N_0,N_1,N_2 \lhd H$ such that $K \cap N_0\subseteq K_0$, and the image of $x$ is not conjugate to the image of  $M_i$ in $H/N_i$ for $i=1,2$.
Set $N'=N_0 \cap N_1\cap N_2$ and $K_1=K \cap N'$. Similarly to Claim I, the homomorphism $\xi:K \to K/K_1$ extends to a homomorphism
$\eta:H \to \bar H$, where $\bar H$ is an HNN-extension of $K/K_1$ with associated subgroups $\xi(M_1)$ and $\xi(M_2)$.

Denote $\bar x =\eta(x)=\xi(x) \in K/K_1\leqslant \bar H$. As before, since $K_1 \leqslant N_i$, we have that $\bar x \notin \xi(M_i)^{K/K_1}$, $i=1,2$, and so we can use
Lemma~\ref{lem:conj_in_HNN} to conclude that $C_{\bar H}(\bar x)=C_{K/K_1}(\bar x)$. Recall that $K_1 \leqslant K_0$, hence the epimorphism from $K$ to $K/K_0$ factors through $\xi$.
Therefore in $\bar H$ we have
\begin{equation}\label{eq:cbh}
C_{\bar H}(\bar x) =C_{K/K_1}(\bar x)\subseteq \xi(C_{K}(x)RK_0) = \xi(C_K(x)R) \subseteq \eta(C_H(x)P),
\end{equation}
because $K_0 \leqslant R\leqslant P$ by construction.

Once again, $\bar H$ is virtually free and so is any subgroup of it. Therefore $\bar P \langle \bar x \rangle\leqslant \bar H$ is conjugacy separable
by Dyer's theorem \cite{Dyer}, where $\bar P=\eta(P)$ is a finite index normal subgroup of $\bar H$.
As above this yields that the subset ${\bar x}^{\bar P \langle \bar x \rangle}={\bar x}^{\bar P}$ is separable in $\bar H$. By Lemma \ref{lem:for_a_given_P_sep_cc->CCH}
there exists a finite index normal subgroup $\bar N \lhd \bar H$ such that $\bar N \leqslant \bar P$ and
\begin{equation}\label{eq:cq}
C_{\bar H/\bar N}(\zeta(\bar x))\subseteq \zeta\left(C_{\bar H}(\bar x)\bar P\right),
\end{equation}
where $\zeta:\bar H \to \bar H/\bar N$ is the natural epimorphism.

Let $N=\eta^{-1}(\bar N)$ be the full preimage of $\bar N$ in $H$, and let $\psi:H \to H/N$ be the natural homomorphism.
Then $\psi=\zeta \circ \eta$ and $\bar H /\bar {N}$ can be identified with $H/N$.
A combination of \eqref{eq:cq} and \eqref{eq:cbh} gives the following inclusion in $H/N$:
\begin{equation}\label{eq:ch}
    C_{H/N}(\psi(x)) \subseteq  \zeta\left(C_{\bar H}(\bar x)\bar P\right) \subseteq  \zeta\left( \eta(C_H(x)P) \bar P\right)=\psi (C_H(x)P).
\end{equation}

To finish the proof of Claim II it remains to show that $N \leqslant P$.
Since $\eta(N)=\bar N \leqslant \bar P=\eta(P)$, it is enough to prove that $\ker\eta\leqslant P$. To this end, observe that  $\ker\eta$ is the normal closure
of $K_1=\ker\xi$ in $H$ (this easily follows from the universal property of HNN-extensions and is left as an exercise for the reader).
Since $K_1\leqslant K_0 \leqslant R \leqslant P$ and $P\lhd H$, we see that the normal closure of $K_1$ in $H$ must also be contained in $P$. Thus $\ker\eta \leqslant P$,
implying that $N\leqslant P$, which finishes the proof of Claim II.

In order to apply Lemma~\ref{lem:CCH->sep_c_c_for_sbgps} we should also note that the subset $C_H(x)H_1$ splits in a finite union of left cosets modulo $H_1$ in $H$
because $|H:H_1|<\infty$, and hence this subset is separable in $H$. In view of Claims I, II we see that all of the assumptions of Lemma~\ref{lem:CCH->sep_c_c_for_sbgps}
are satisfied. Therefore $x^{H_1}$ is separable in $H$, and the consideration of Case 2 is finished.

Thus we have shown that $x^{H_1}$ is separable in $H$ for all $x \in H$ and any finite index subgroup $H_1 \leqslant H$. Since the profinite topology of a subgroup
is finer than the topology induced from the ambient group, we can conclude that $x^{H_1}$ is separable in $H_1$ whenever $x \in H_1$. Consequently
$H_1$ is conjugacy separable. Since $H_1$ was chosen as an arbitrary finite index subgroup of $H$, we see that $H$ is hereditarily conjugacy separable.
\end{proof}

\begin{proof}[Proof of Corollary \ref{cor:qc-cs}] Let $H \leqslant G$ be a quasiconvex subgroup. By Newman's theorem \cite[Thm. 2]{Newman}, for any $x \in H\setminus \{1\}$
 there is $g \in G$ such that $C_G(x)=\langle g \rangle$. Hence $x=g^k \in H$ for some $k \in \N$ and so the subset $C_G(x)H$ splits in a finite union of left cosets
 modulo $H$. Now, since $G$ is virtually compact special by \cite[Cor. 18.3]{Wise-qc-h}, quasiconvex subgroups are separable in $G$ by \cite[Thm. 7.3, Lemma 7.5]{H-W-1}.
 It follows that $H$ and, hence,  $C_G(x)H$ are separable in $G$, for an arbitrary $x \in H$ (if $x=1$ then $C_G(x)H=G$).

By Theorem \ref{thm:main}, $G$ is hereditarily conjugacy separable and so every element $x \in G$ satisfies ${\rm CC}_G$ (see \cite[Prop. 3.2]{M-RAAG}).
Therefore we can apply Lemma \ref{lem:CCH->sep_c_c_for_sbgps} to conclude that $x^H$ is separable in $G$ (and, hence, in $H$). Thus $H$ is conjugacy separable, as claimed.
\end{proof}

\end{document}